\documentclass{amsart}
\usepackage[utf8]{inputenc}

\usepackage{hyperref}
\usepackage{amsmath}
\usepackage{amsmath,amsthm,amssymb,amsfonts}
\usepackage[utf8]{inputenc}
\usepackage{enumerate}
\usepackage{graphicx}
\usepackage{enumitem}
\usepackage{tikz}
\usepackage[numbers]{natbib}

\theoremstyle{plain}
\newtheorem{theorem}{Theorem}[section]
\newtheorem{lemma}[theorem]{Lemma}
\newtheorem{claim}[theorem]{Claim}
\newtheorem{corollary}[theorem]{Corollary}
\newtheorem{proposition}[theorem]{Proposition}

\newtheorem{question}[theorem]{Question}

\newtheorem{fact}[theorem]{Fact}

\theoremstyle{definition}
\newtheorem{definition}[theorem]{Definition}

\newtheorem{remark}[theorem]{Remark}

\newcommand{\Q}{\mathbb{Q}}
\newcommand{\R}{\mathbb{R}}

\newcommand{\Z}{\mathbb{Z}}
\DeclareMathOperator{\hp}{Homeo^+}

\DeclareMathOperator{\fix}{Fix}

\DeclareMathOperator{\aut}{Aut}
\DeclareMathOperator{\id}{id}
\DeclareMathOperator{\conv}{conv}

\begin{document}
  
  \title{The structure of random automorphisms of the rational numbers}
	\author[U. B. Darji]{Udayan B. Darji}
	\address{Department of Mathematics, University of Louisville,
		Louisville, KY 40292, USA\\Ashoka University, Rajiv Gandhi Education City, Kundli, Rai 131029, India} 
	\email{ubdarj01@louisville.edu}
	\urladdr{http://www.math.louisville.edu/\!$\tilde{}$ \!\!darji}
	\author[M. Elekes]{M\'arton Elekes}
	\address{Alfr\'ed R\'enyi Institute of Mathematics, Hungarian Academy of Sciences,
		PO Box 127, 1364 Budapest, Hungary and E\"otv\"os Lor\'and
		University, Institute of Mathematics, P\'azm\'any P\'eter s. 1/c,
		1117 Budapest, Hungary}
	\email{elekes.marton@renyi.mta.hu}
	\urladdr{www.renyi.hu/ \!$\tilde{}$ \!\!emarci}
	
	\author[K. Kalina]{Kende Kalina}
		\address{E\"otv\"os Lor\'and
		University, Institute of Mathematics, P\'azm\'any P\'eter s. 1/c,
		1117 Budapest, Hungary}
       \email{kkalina@cs.elte.hu} 
	\author[V. Kiss]{Viktor Kiss}
	\address{Alfr\'ed R\'enyi Institute of Mathematics, Hungarian Academy of Sciences,
		PO Box 127, 1364 Budapest, Hungary and E\"otv\"os Lor\'and
		University, Institute of Mathematics, P\'azm\'any P\'eter s. 1/c,
		1117 Budapest, Hungary}
	\email{kiss.viktor@renyi.mta.hu}
	
	\author[Z. Vidny\'anszky]{Zolt\'an Vidny\'anszky}

	\address{Kurt Gödel Research Center for Mathematical Logic, Universität Wien, Währinger Strasse 25, 1090 Wien, Austria and Alfr\'ed 	R\'enyi Institute of Mathematics, Hungarian Academy of Sciences,
		PO Box 127, 1364 Budapest}
 	\email{zoltan.vidnyanszky@univie.ac.at}
  \urladdr{www.logic.univie.ac.at/$\tilde{}$ vidnyanszz77}

  
	\thanks{The second, fourth and fifth authors were partially supported by the
		National Research, Development and Innovation Office
		-- NKFIH, grants no.~113047, no.~104178 and no.~124749. The fifth author was also partially supported by FWF Grant P29999.}
  \begin{abstract}
    In order to understand the structure of the ``typical'' element of an automorphism group, one has to study how large the conjugacy classes of the group are. For the case when typical is meant in the sense of Baire category, Truss proved that there is a co-meagre conjugacy class in $\aut(\Q, <)$, the automorphism group of the rational numbers. Following Dougherty and Mycielski we investigate the measure theoretic dual of this problem, using Christensen's notion of Haar null sets. We give a complete description of the size of the conjugacy classes of $\aut(\Q, <)$ with respect to this notion. In particular, we show that there exist continuum many non-Haar null conjugacy classes, illustrating that the random behaviour is quite different from the typical one in the sense of Baire category.
  \end{abstract}

  \maketitle
  
\section{Introduction}
  
  The study of typical elements of Polish groups is a flourishing field with a large number of applications. The systematic investigation of typical elements of automorphism groups of countable structures was initiated by Truss \cite{truss1992generic}, in particular, he proved that there is a co-meagre conjugacy class in $\aut(\Q, <)$, the automorphism group of the rational numbers.  Kechris and Rosendal \cite{KechrisRosendal} characterised the existence of a co-meagre conjugacy class in model theoretic terms, and investigated the relation between the existence of co-meagre conjugacy classes in every dimension and other group theoretic properties, such as the small index property, uncountable cofinality, automatic continuity and Bergman's property. 
  
  Thus, it is natural to ask whether there exist measure theoretic analogues of these results. Unfortunately, on non-locally compact groups there is no natural invariant $\sigma$-finite measure. However, a generalisation of the ideal of measure zero sets can be defined in every Polish group as follows:
  
  \begin{definition}[Christensen, \cite{originalhaarnull}]
    \label{d:haarnull}
    Let $G$ be a Polish group and $B \subset G$ be Borel. We say that $B$ is \textit{Haar null} if there exists a Borel probability measure $\mu$ on $G$ such that for every $g,h \in G$ we have $\mu(gBh)=0$. 
    An arbitrary set $S$ is called Haar null if $S \subset B$ for some Borel Haar null set $B$.
  \end{definition}
  
  It is known that the collection of Haar null sets forms a $\sigma$-ideal in every Polish group and it coincides with the ideal of measure zero sets in locally compact groups with respect to every left (or equivalently right) Haar measure. Using this definition, it makes sense to talk about the properties of random elements of a Polish group. A property $P$ of elements of a Polish group $G$ is said to \textit{hold almost surely} or \emph{almost every element of G has property $P$} if the set $\{g \in G: g \text{ has property } P\}$ is co-Haar null. 
  
  Dougherty and Mycielski \cite{DM} used the notion of Haar null sets to examine $S_\infty$, the permutation group of the countably infinite set. They gave a complete description of the non-Haar null conjugacy classes and the (conjugacy invariant) properties of a random element of $S_\infty$. It follows from their characterisation that almost every element in $S_\infty$ has finitely many finite and infinitely many infinite orbits, where a set of the form $\{g^n(x) : n \in \Z\}$ is called an \emph{orbit} of $g$. In \cite{auto}, the authors of the current paper initiated a study of the size of the conjugacy classes of automorphism groups of general countable structures. In particular, they characterised those automorphism groups in which almost every element has finitely many finite and infinitely many infinite orbits. 
  
  Their results, however, fall short of giving a complete description of the size of the conjugacy classes of automorphism groups. The aim of the current paper is to extend the results of \cite{auto} for $\aut(\Q, <)$, the automorphism group of the rational numbers, by characterising the non-Haar null conjugacy classes. 
  We would like to point out that a similar characterisation result can
  be proved for $\aut(\mathcal{R})$, the automorphism group of the random graph (see \cite{autrcikk}). Interestingly, the proof is completely different, hence
  the following question is very natural:
  
  \begin{question}
    Is it possible to unify these proofs? Are there necessary and sufficient model theoretic conditions which characterise the measure theoretic behaviour of the conjugacy classes?
  \end{question}
  
  The paper is organised as follows. In Section \ref{s:prel} we collect the notions and facts we will use. We also state a result from \cite{auto}, and prove a generalisation for a result of Christensen that we will need. In Section \ref{s:autQ} we formulate and prove the main result of the paper by characterising the non-Haar null conjugacy classes of $\aut(\Q, <)$.

  \section{Preliminaries and notations}
  \label{s:prel}
  
  We will follow the notations of \cite{kechrisbook}. For a detailed introduction to the theory of Polish groups see \cite[Chapter 1]{becker1996descriptive}, while the model theoretic background can be found in \cite[Chapter 7]{hodges}. Nevertheless, we summarise the basic facts which we will use. 
  
  We think of $S_\infty$ as the permutation group of $\omega$. With the topology of pointwise convergence, it is a Polish group. Let $G$ be a closed subgroup of $S_\infty$. The \emph{orbit} of an element $x \in \omega$ (with respect to $g \in G$) is the set $\{g^n(x) : n \in \Z\}$. For a set $S \subset \omega$ we denote the \emph{pointwise stabiliser} of $S$ by $G_{(S)}$, that is, $G_{(S)} = \{g \in G : \forall s \in S\; (g(s) = s)\}$. In case $S = \{x\}$, we write $G_{(x)}$ instead of $G_{(\{x\})}$. As usual, we use the notation $G(x) = \{g(x) : g \in G\}$
  
  \begin{definition} \label{NACdef} 
    Let $G$ be a closed subgroup of $S_\infty$. We say that \textit{$G$ has the finite algebraic closure property ($FACP$)} if for every finite $S \subset \omega$ the set $\{x:|G_{(S)}(x)|<\infty\}$ is finite.
  \end{definition}
  
  Generalizing the results of Dougherty and Mycielski, the following theorem is proved in \cite{auto}.
  
  \begin{theorem}
    \label{t:FACP <-> finfin -> infinf}
    Let $G \le S_\infty$ be closed. Then the following are equivalent:
    \begin{enumerate}
      \item \label{tc:finfin} almost every element of $G$ has finitely many finite orbits,
      \item \label{tc:FACP} $G$ has the $FACP$.
    \end{enumerate}
    Moreover, any of the above conditions implies that almost every element of $G$ has infinitely many infinite orbits.
  \end{theorem}
  
  Note that the automorphism group of a countable structure can be considered to be a subgroup of $S_\infty$ by identifying the domain of the structure with (a subset of) $\omega$. It is a well-known fact that such an automorphism group then becomes a closed subgroup of $S_\infty$. Applying this fact to $\aut(\Q, <)$ and noticing that it has the $FACP$, we obtain the following corollary.
  \begin{corollary}
    \label{c:fininf}
    In $\aut(\mathbb{Q}, <)$ almost every element has finitely many finite and infinitely many infinite orbits.
  \end{corollary}

  Let us consider the following notion of largeness:
  
  \begin{definition}
    \label{d:catcherbiter}
    Let $G$ be a Polish topological group. A set $A \subset G$ is called \textit{compact catcher} if for every compact $K \subset G$ there exist $g,h \in G$ so that $gKh \subset A$.  $A$ is \textit{compact biter} if for every non-empty compact $K \subset G$ there exist an open set $U$ and $g,h \in G$ so that $U \cap K \not = \emptyset$, and $g(U \cap K)h \subset A$. 
  \end{definition}
  The following well-known observation is one of the most useful tools to prove that a certain set is not Haar null (for a proof, see e.g. \cite{auto}).
  \begin{fact}
    \label{f:biter}
    If $A$ is compact biter then it is not Haar null.
  \end{fact}
  
  We remark here that the proof of Dougherty and Mycielski about $S_\infty$ actually shows that every non-Haar null conjugacy class is compact biter and the unique non-Haar null conjugacy class which contains elements without finite orbits is compact catcher. We prove a similar statement about $\aut(\Q, <)$ in Section \ref{s:autQ}.
  
  It is sometimes useful to consider right and left Haar null sets: a Borel set $B$ is \emph{right (resp. left) Haar null} if there exists a Borel probability measure $\mu$ on $G$ such that for every $g \in G$
  we have $\mu(Bg)=0$ (resp. $\mu(gB)=0$). An arbitrary set $S$ is called \emph{right (resp. left) Haar null} if $S \subset B$ for some Borel right (resp. left) Haar null set $B$. The following observation will be used several times (for a proof, see e.g. \cite{auto}).
  
  \begin{lemma}
    \label{l:conjugacyinvariant}
    Suppose that $B$ is a Borel set that is invariant under conjugacy. Then $B$ is left Haar null iff it is right Haar null iff it is Haar null.
  \end{lemma}
  
  We will use the following well-known fact (for the case $G = \aut(\Q, <)$) throughout the paper. 
  \begin{fact}
    Suppose that $G \le S_\infty$ is a closed subgroup, and $K \subseteq G$. Then $K$ is compact if and only if for each $x \in \omega$ the sets $\{g(x) : g \in K\}$ and $\{g^{-1}(x) : g \in K\}$ are finite. 
  \end{fact}
 
\subsection{Christensen's theorem revisited}
  \label{s:chris}
  
  We will need a straightforward generalisation of a theorem proved by Christensen \cite{originalhaarnull}, here we reiterate Rosendal's proof (see \cite{zrosendal}).
  
  \begin{theorem} (Christensen) 
    \label{t:null}
    Let $A \subset G$ be a conjugacy invariant set and suppose that there exists a cover of $A$ by Borel sets $A=\bigcup_{n \in \omega} A_n$ (in particular, $A$ is also a Borel set) and a conjugacy invariant set $B$ so that $1 \in \overline{B}$ and $B \cap \bigcup_{n \in \omega} A_n^{-1}A_n=\emptyset$. Then $A$ is Haar null.
  \end{theorem}
  
  \begin{remark}
    Some authors use a definition for Haar null sets which slightly differs from Definition \ref{d:haarnull}. Namely, according to that version, a set $S$ is Haar null, if there exists a
    Borel probability measure $\mu$ on $G$ and a \emph{universally measurable set} $U$ such that $S\subset U$ and for every $g,h \in G$ we have $\mu(gUh)=0$. These two notions differ in general (see \cite{gdeltahull}), although they coincide for analytic sets (see \cite{soleckion}). We would like to point out that the above theorem and Corollary \ref{c:christensen} remain true (and can be proved in the same way) if we change Borel to universally measurable everywhere and use the mentioned alternative definition of Haar null sets.     
  \end{remark}
  \begin{proof}[Proof of Theorem \ref{t:null}]
    We claim that there exists a sequence $\{g_i:i \in \omega\} \subset B$ with $g_i \to 1$ and the following properties: 
    
    \begin{itemize}
      \item for every  $(\varepsilon_i)_{i \in \omega} \in 2^\omega$ we have that the sequence $(g^{\varepsilon_0}_0g^{\varepsilon_1}_1\dots g^{\varepsilon_n}_n)_{n \in \omega}$ converges,
      \item the map $\phi:2^\omega \to G$ defined by $(\varepsilon_n)_{ n \in \omega} \mapsto g^{\varepsilon_0}_0g^{\varepsilon_1}_1 g^{\varepsilon_2}_2\dots$ is continuous (the right hand side expression makes sense because of the convergence).
    \end{itemize}
    We can choose such a sequence by induction: fix a compatible complete metric and suppose that we have already selected $g_0,g_1,\dots ,g_n$. Now notice that for every $(\varepsilon_0,\dots,\varepsilon_n) \in 2^{n+1}$ the set $\{x \in G:d(g^{\varepsilon_0}_0g^{\varepsilon_1}_1\dots g^{\varepsilon_n}_nx,g^{\varepsilon_0}_0g^{\varepsilon_1}_1\dots g^{\varepsilon_n}_n)<2^{-n-1}\}$ contains a neighbourhood of the identity. Therefore we can choose a 
    \[g_{n+1} \in B \cap \bigcap_{(\varepsilon_0,\dots,\varepsilon_n) \in 2^{n+1}} \{x \in G:d(g^{\varepsilon_0}_0g^{\varepsilon_1}_1\dots g^{\varepsilon_n}_nx,g^{\varepsilon_0}_0g^{\varepsilon_1}_1\dots g^{\varepsilon_n}_n)<2^{-n-1}\}.\]
    One can easily show that for every $(\varepsilon_n)_{ n \in \omega} \in 2^\omega$ the sequence $(g^{\varepsilon_0}_0g^{\varepsilon_1}_1\dots g^{\varepsilon_n}_n)_{n \in \omega}$ is Cauchy and the function $\phi$ is continuous.
    
    Let $\lambda$ be the usual measure on $2^\omega$ and let $\lambda_*=\phi_*\lambda$, its push forward. We claim that $\lambda_*$ witnesses that $A$ is left-Haar null which is equivalent to its Haar nullness, by the fact that $A$ is conjugacy invariant and by using Lemma \ref{l:conjugacyinvariant}. 
    
    Suppose not, then there exists an $f \in G$ so that $\lambda_*(fA)>0$, therefore $\lambda_*(fA_k)>0$ for some $k \in \omega$. This is equivalent to $\lambda(\phi^{-1}(fA_k))>0$ and if we regard $2^\omega$ as $\Z_2^\omega$, by Weil's theorem (see e.g. \cite{zrosendal}) we have that $\phi^{-1}(fA_k)-\phi^{-1}(fA_k)$ contains a neighbourhood of $(0,0,\dots)$, the identity in $\Z_2^\omega$. Then there exists an element in $\phi^{-1}(fA_k)-\phi^{-1}(fA_k)$ which is zero at every coordinate except for one. Thus, $\phi^{-1}(fA_k)$ contains two elements of the form $(\varepsilon_0,\dots,\varepsilon_{n-1},0,\varepsilon_{n+1},\dots)$ and  $(\varepsilon_0,\dots,\varepsilon_{n-1},1,\varepsilon_{n+1},\dots)$, i. e., differing at exactly one place. Then taking the $\phi$ images of these elements we obtain that there exist $h_1,h_2 \in G$ so that $h_1 h_2 \in fA_k$ and $h_1 g_n h_2 \in fA_k$. This implies \[h^{-1}_2 h^{-1}_1 h_1 g_n h_2 \in A^{-1}_kA_k\]
    thus 
    \[h^{-1}_2 g_n  h_2 \in A^{-1}_kA_k\]
    but by the conjugacy invariance of $B$ we get
    \[h^{-1}_2 g_n  h_2 \in B \cap A^{-1}_kA_k,\]
    contradicting the initial assumptions of the theorem.
  \end{proof}
  
  Letting $A=A_n$  for every $n \in \omega$ and using that if $A$ is conjugacy invariant then so is $G \setminus A^{-1}A$ we can deduce the following corollary:
  \begin{corollary}
    \label{c:christensen}
    If $A$ is a conjugacy invariant Borel set that is not Haar null then $A^{-1}A$ contains a neighbourhood of the identity. 
  \end{corollary}
  
\section{Characterisation of the non-Haar null conjugacy classes of \texorpdfstring{$\aut(\Q, <)$}{Aut(Q, <)}} 
  \label{s:autQ}
  
  We now formulate and prove our main theorem. 
  In the investigation of the structure of $\aut(\Q, <)$ we use the concept of orbitals (defined below, for more details on this topic see \cite{Glass}). Let $p, q \in \Q$. The interval $(p, q)$ will denote the set $\{r \in \Q : p < r < q\}$. For an automorphism $f \in \aut(\Q, <)$, we denote the set of fixed points of $f$ by $\fix(f)$. 
  \begin{definition}
    \label{d:orbital}
    The set of \emph{orbitals} of an automorphism $f \in \aut(\Q, <)$, $\mathcal{O}^*_f$, consists of 
    the convex hulls (relative to $\Q$) of the orbits of the rational numbers, that is
    $$
    \mathcal{O}^*_f = \{ \conv(\{f^n(r) : n \in \Z\}) : r \in \Q\}.
    $$
  \end{definition}
  It is easy to see that the orbitals of $f$ form a partition of $\Q$, with the fixed points determining one element orbitals, hence ``being in the same orbital'' is an equivalence relation. Using this fact, we define the relation $<$ on the set of orbitals by letting $O_1 < O_2$ for distinct $O_1, O_2 \in \mathcal{O}^*_f$ if $p_1 < p_2$ for some (and hence for all) $p_1 \in O_1$ and $p_2 \in O_2$. Note that $<$ is a linear order on the set of orbitals. 
  
  It is also easy to see that if $p, q \in \Q$ are in the same orbital of $f$ then $f(p) > p \Leftrightarrow f(q) > q$, $f(p) < p \Leftrightarrow f(q) < q$ and $f(p) = p \Leftrightarrow f(q) = q \Rightarrow p = q$. This observation makes it possible to define the \emph{parity function}, $s_f : \mathcal{O}^*_f \to \{-1, 0, 1\}$. Let $s_f(O) = 0$ if $O$ consists of a fixed point of $f$, $s_f(O) = 1$ if $f(p) > p$ for some (and hence, for all) $p \in O$ and $s_f(O) = -1$ if $f(p) < p$ for some (and hence, for all) $p \in O$. 

  The following is our main result. 
  
  \begin{theorem}
    \label{t:qintro}
    For almost every element $f$ of $\aut(\mathbb{Q}, <)$ 
    \begin{enumerate}
      \item \label{pt:autq1} for orbitals $O_1, O_2 \in \mathcal{O}^*_f$ with $O_1 < O_2$ such that $s_f(O_1) = s_f(O_2) = 1$ or $s_f(O_1) = s_f(O_2) 
      = -1$, there exists an orbital $O_3 \in \mathcal{O}^*_f$ with $O_1 < O_3 < O_2$ 
      and $s_f(O_3) \neq s_f(O_1)$,
      \item \label{pt:autq2} $f$ has only finitely many fixed points. 
    \end{enumerate}
    These properties characterise the non-Haar null conjugacy classes, that is, a conjugacy class is non-Haar null if and only if one (or equivalently each) of its elements has properties \eqref{pt:autq1} and \eqref{pt:autq2}.
    
    Moreover, every non-Haar null conjugacy class is compact biter and those non-Haar null classes in which the elements have no rational fixed points are compact catchers. 
  \end{theorem}

  For a subset $F \subset \R \setminus \Q$ that is closed in $\R$ one can construct an automorphism $f \in \aut(\Q, <)$ such that the set of fixed points of the unique extension of $f$ to a homeomorphism of $\R$ is $F$, and $f$ satisfies conditions \eqref{pt:autq1} and \eqref{pt:autq2}. It is not hard to see that one can find continuum many pairwise non-order isomorphic such closed sets, producing non-conjugate automorphisms. Hence we obtain the following surprising corollary:
  
  \begin{corollary}
    \label{c:autqcont}
    There are continuum many non-Haar null conjugacy classes in $\aut(\mathbb{Q},<)$, and their union is co-Haar null.
  \end{corollary}

  Note that it was proved by Solecki \cite{openlyhaarnull} that in every non-locally compact Polish group that admits a two-sided invariant metric there are continuum many pairwise disjoint non-Haar null Borel sets, thus the above corollary is an extension of his results for $\aut(\mathbb{Q},<)$. We would like to point out that in a sharp contrast to this result, in $\hp([0,1])$ (that is, in the group of order preserving homeomorphisms of the interval) the random behaviour is quite different (see \cite{homeo}), more similar to the case of $S_\infty$: there are only countably many non-Haar null conjugacy classes and their union is co-Haar null.

  Instead of proving Theorem \ref{t:qintro}, we prove the following two theorems separately. It is straightforward to check that these theorems indeed imply Theorem \ref{t:qintro}.
  
  \begin{theorem}
    \label{t:autQ}
    The conjugacy class of $f \in \aut(\Q, <)$ is non-Haar null if and only if 
    $\fix(f)$ is finite, and for each pair of orbitals $O_1, O_2 \in \mathcal{O}^*_f$ 
    with $O_1 < O_2$ such that $s_f(O_1) = s_f(O_2) = 1$ or $s_f(O_1) = s_f(O_2) 
    = -1$, there exists an orbital $O_3 \in \mathcal{O}^*_f$ with $O_1 < O_3 < O_2$ 
    and $s_f(O_3) \neq s_f(O_1)$. 
    
    Moreover, every non-Haar null conjugacy class is compact biter and those non-Haar null classes in which the elements have no fixed points are compact catchers.
  \end{theorem}
  
  \begin{theorem}
    \label{t:autQ union of null}
    The union of the Haar null conjugacy classes of $\aut(\Q, <)$ is Haar null.
  \end{theorem}
  
  We say that an automorphism is \emph{good} if it satisfies the conditions of 
  Theorem \ref{t:qintro}. In the proof of Theorems \ref{t:autQ} we use the following lemma to check 
  conjugacy between good automorphisms.
  
  \begin{lemma}
    \label{l:conjugacy in aut(Q)}
    Let $f$ and $g$ be good automorphisms without fixed points. Suppose that there exists a function $\phi: \Q \to \mathcal{O}^*_f$ with the following properties: it is monotonically increasing (not necessarily strictly), surjective, and for each $q \in \Q$, 
    \begin{enumerate}[label=(\arabic*)]
      \item\label{il:g(q) > q <=> s(phi(q)) = 1} $g(q) > q \Leftrightarrow   
      s_f(\phi(q)) = 1$;
      \item\label{il:g(q) < q <=> s(phi(q)) = -1} $g(q) < q \Leftrightarrow 
      s_f(\phi(q)) = -1$.
    \end{enumerate}
    Then $f$ and $g$ are conjugate automorphisms. 
  \end{lemma}
  \begin{proof}
    We use the characterization in \cite[Theorem 2.2.5]{Glass} to check the conjugacy of 
    automorphisms: $f$ and $g$ are conjugate if and only if there exists 
    an order preserving bijection $\psi : \mathcal{O}^*_g \to \mathcal{O}^*_f$ such 
    that $s_g(O) = s_f(\psi(O))$ for every $O \in \mathcal{O}^*_g$. 
    
    We now show that it is legal to define the appropriate bijection $\psi$ as 
    $\psi(O) = O'$ where $O' = \phi(p)$ for some $p \in O$. To show that it is a 
    well-defined map, we need to prove that given $O \in \mathcal{O}^*_g$ and 
    $p, q \in O$, $\phi(p) = \phi(q)$. Suppose the contrary and note that $s_g(O) = 1$ or $s_g(O) = -1$, since $g$ is fixed-point free. We now suppose that $s_g(O) = 1$, the case where $s_g(O) = -1$ is analogous. Then $g(p) > p$ and $g(q) > q$, hence 
    $s_f(\phi(p)) = s_f(\phi(q)) = 1$. Since $f$ is good, and $\phi(p) \neq 
    \phi(q)$ by our assumption, there is an orbital $O' \in \mathcal{O}^*_f$ such 
    that $\phi(p) < O' < \phi(q)$ and $s_f(O') \neq 1$. Using that $\phi$ is 
    surjective and monotone increasing, there exists an $r \in (p, q)$ such that 
    $\phi(r) = O'$. Then $s_f(\phi(r)) \neq 1$, but $r$ is in the same orbital as 
    $p$ and $q$, since orbitals are convex, hence $g(r) > r$. This contradicts 
    \ref{il:g(q) > q <=> s(phi(q)) = 1}. 
    
    The map $\psi$ is increasing and surjective, since $\phi$ is increasing 
    and surjective. One can easily check that conditions \ref{il:g(q) > q <=> s(phi(q)) = 1} and \ref{il:g(q) < q <=> s(phi(q)) = -1} imply that for every $O \in \mathcal{O}^*_g$, $s_g(O) = s_f(\psi(O))$. 
    Hence it remains to show that $\psi$ is injective. 
    
    Let $O, O' \in \mathcal{O}^*_g$ be distinct orbitals with $\psi(O) = \psi(O')$. 
    Using conditions \ref{il:g(q) > q <=> s(phi(q)) = 1} and \ref{il:g(q) < q <=> s(phi(q)) = -1}, we have $s_g(O) = s_g(O')$. If $s_g(O) = s_g(O') = 1$ then using that $g$ is good, there exists an 
    orbital $O'' \in \mathcal{O}^*_g$ between $O$ and $O'$ such that $s_g(O'') \neq 
    1$. Then using the monotonicity of $\psi$ one 
    obtains $\psi(O'') = \psi(O')$, hence $1 \neq s_f(\psi(O'')) = s_f(\psi(O)) = 1$, a contradiction. An analogous argument shows that $s_g(O) = s_g(O') = -1$ also leads to a contradiction, hence the proof of the lemma is complete.
  \end{proof}
  
  The proof that the conjugacy class of a good automorphism is compact biter, rests on a technical construction. We first deal with the case when the automorphism has no fixed points; the general statement will follow easily in the proof of Theorem \ref{t:autQ}. 
  
  \begin{proposition}
    \label{p:cpct translated into conj class of fixed-point free}
    Let $f \in \aut(\Q, <)$ be a good automorphism without fixed points, $\mathcal{C}$ be the conjugacy class of $f$ and $\mathcal{K} \subset \aut(\Q, <)$ be compact. Then there is an automorphism $g \in \aut(\Q, <)$ such that $g^{-1} \mathcal{K} \subseteq \mathcal{C}$. 
  \end{proposition}
  
  Along with $g$, we construct a function $\phi : \Q \times \mathcal{K} \to \mathcal{O}^*_f$, recursively. For an automorphism $h \in \mathcal{K}$, $\phi$ will keep track of the orbital of a point $r$ under $g^{-1}h$. Let $r_1, r_2, \dots$ be an enumeration of $\Q$. Let $O_1, O_2, \dots$ be an infinite sequence of elements of $\mathcal{O}^*_f$ containing every such element at least once. Note that $\mathcal{O}^*_f$ may be finite, hence the sequence may contain the same element more than once. At the $n$th step of the recursive construction, we have a finite set $H_n \subset \Q$ and functions $g_n$ and $\phi_n$ (that are partial approximations of the final $g$ and $\phi$). We preserve the following properties of these sets and functions: 
  
  For every $n \in \omega$, $h, h_1, h_2 \in \mathcal{K}$ and $p, p', p'' \in 
  H_n$, where $p' < p''$ and $(p', p'') \cap H_n = \emptyset$, 
  
  \begin{enumerate}[label=(\roman*)]
    \item\label{i:extension} $H_0 \subset H_1 \subset \dots$, $g_0 \subset g_1 
    \subset \dots$ and $\phi_0 \subset \phi_1 \subset \dots$;
    \item\label{i:H_n finite} $H_n \subset \Q$ is finite;
    \item\label{i:g_n increasing} $g_n : H_n \to \Q$ is 
    strictly increasing;
    \item\label{i:phi increasing} $\phi_n : H_n \times \mathcal{K} \to 
    \mathcal{O}^*_f$, and $\phi_n(., h)$ is increasing;
    \item\label{i:surjective} $r_1, \dots, r_{n + 1} \in H_{3n + 1} \cap g_{3n + 2}(H_{3n + 2})$ and $O_1, \dots, O_{n + 1} \in \phi_{3n + 3}(H_{3n + 3}, h)$;
    \item\label{i:no crossing} it cannot happen that 
    $h_1(p') < g_n(p') < h_2(p')$, $h_1(p'') > g_n(p'') > h_2(p'')$;
    \item\label{i:extension is good} if $h(p') > g_n(p')$ and $h(p'') > 
    g_n(p'')$ then $h(r) \ge g_n(p'')$ for every $r \in [p', p'']$; 
    similarly, if $h(p') < g_n(p')$ and $h(p'') < g_n(p'')$ then 
    $h(r) \le g_n(p')$ for every $r \in [p', p'']$ (thus extending $g_n$ in 
    any way to a strictly increasing function on $[p', p'']$, there is no $r 
    \in [p', p'']$ where the value of the extension can be equal to $h(r)$);
    \item\label{i:s=1 <=> g < h, s=-1 <=> g > h} $s_f(\phi_n(p, h)) = 1 
    \Leftrightarrow g_n(p) < h(p)$ and $s_f(\phi_n(p, h)) = -1 \Leftrightarrow 
    g_n(p) > h(p)$;
    \item\label{i:phi image is same} $\phi_n(H_n, h_1) = \phi_n(H_n, h_2)$;
    \item\label{i:s is alternating} the value of $s_f$ is alternating on 
    the image $\phi_n(H_n, h)$, i.e., either $\phi_n(p', h) = \phi_n(p'', 
    h)$ or $s_f(\phi_n(p', h)) \neq s_f(\phi_n(p'', h))$;
    \item\label{i:change for same place only} $h_i(p') > g_n(p')$ and 
    $h_i(p'') < g_n(p'')$ ($i = 1, 2$) (or similarly, $h_i(p') < g_n(p')$ 
    and $h_i(p'') > g_n(p'')$ ($i = 1, 2$)) implies that $\phi_n(p', h_1) = 
    \phi_n(p', h_2)$ and $\phi_n(p'', h_1) = \phi_n(p'', h_2)$.
  \end{enumerate}
  
  \begin{remark}
    Conditions \ref{i:no crossing} and \ref{i:extension is good} are 
    equivalent to the following fact: the rectangle $(p', p'') \times (g_n(p'), g_n(p''))$ has two sides that are 
    opposite such that no $h \in \mathcal{K}$ intersects the interior of any of 
    those sides. 
  \end{remark}
  
  First we show that the inductive construction can be carried out satisfying the conditions.
  
  \begin{claim}
    The sets and functions $H_n$, $g_n$ and $\phi_n$ can be constructed with the above properties. 
  \end{claim}
  \begin{proof}
    We prove the claim by induction on $n$. For $n = 0$, let $H_0 = g_0 = \phi_0 = \emptyset$. Now suppose that $H_n$, $g_n$ and $\phi_n$ are given with the above properties, using them, we construct the suitable $H_{n + 1}$, $g_{n + 1}$ and $\phi_{n + 1}$. There are three cases according to the remainder of $n$ mod $3$: if the remainder is $0$, we add the next rational number (according to the enumeration $(r_i)$) to $H_{n + 1}$, and hence to the domain of $g_{n + 1}$. If the remainder is $1$, we add the next rational number to the range of $g_{n + 1}$. If the remainder is $2$, we make sure that the next orbital (according to the enumeration $(O_i)$) is in the range of $\phi(., h)$ for every $h \in \mathcal{K}$. 
    
    \bigskip
    {\bf Case 1: $n = 3m$.}  At this step, we make sure that $r_{m + 1} \in H_{n + 
      1}$. If already $r_{m + 1} \in H_n$ then let $H_{n + 1} = H_n$, $g_{n + 1} 
    = g_n$ and $\phi_{n + 1} = \phi_n$. Otherwise, there are multiple cases 
    according to the existence of $p' \in H_n$ with $p' < r_{m + 1}$, $p'' \in 
    H_n$ with $r_{m + 1} < p''$, and whether $g_n(p') < h(p')$ or $g_n(p') > 
    h(p')$, $g_n(p'') < h(p'')$ or $g_n(p'') > h(p'')$. 
    
    Case 1a: there are neither $p' \in H_n$ with $p' < r_{m + 1}$ nor $p'' \in 
    H_n$ with $r_{m + 1} < p''$ (that is, $H_n = \emptyset$, $n = 0$). If 
    $s_f(O_1) = 1$ then we find $q \in \Q$ with $q < h(r_{m + 1})$ 
    for every $h \in \mathcal{K}$, otherwise, we find $q \Q$ 
    with $q > h(r_{m + 1})$ for every $h \in \mathcal{K}$. Such a $q$ exists, 
    since $\mathcal{K}$ is compact, thus $\{h(r_{m + 1}) : h \in \mathcal{K}\}$ 
    is finite. Now we set $H_{n + 1} = \{r_{m + 1}\}$, $g_{n + 1}(r_{m + 1}) = 
    q$, $\phi_{n + 1}(r_{m + 1}, h) = O_1$ for every $h \in \mathcal{K}$. 
    
    Case 1b: there is a $p' \in H_n$ with $p' < r_{m + 1}$ but there is no $p'' 
    \in H_n$ with $r_{m + 1} < p''$. Let $p'$ be the largest element in $H_n$, 
    clearly $p' < r_{m + 1}$. Let $q' = g_n(p')$. Using \ref{i:phi image is 
      same}, $\phi_n(p', h)$ is the same for every $h \in \mathcal{K}$, since it 
    is the largest element in the common image $\phi_n(H_n, h)$. Let $O = 
    \phi_n(p', h)$ for some $h \in \mathcal{K}$. 
    Depending on $s_f(O)$, $g_n(p') < h(p')$ for every $h \in \mathcal{K}$ or 
    $g_n(p') > h(p')$ for every $h \in \mathcal{K}$ using \ref{i:s=1 <=> g < 
      h, s=-1 <=> g > h}. In the first case, choose $t \in \Q$ such 
    that $q' < t < h(p')$ for every $h \in \mathcal{K}$. 
    Then set $H_{n + 1} = H_n \cup \{r_{m + 1}\}$ and let $g_{n + 1}$ extend 
    $g_n$ with $g_{n + 1}(r_{m + 1}) = t$, and let $\phi_{n + 1}$ extend 
    $\phi_n$ with $\phi_{n + 1}(r_{m + 1}, h) = O$ for every $h \in 
    \mathcal{K}$. 
    
    In the second case, let $h(r_{m + 1}) < t$ for every $h \in \mathcal{K}$, 
    also satisfying $t > q'$. Choose $q \in (q', t)$ such 
    that $q > h(r_{m + 1})$ for every $h \in \mathcal{K}$. As 
    $h^{-1}(q') > p'$ for every $h \in \mathcal{K}$, there exists $p \in (p', 
    r_{m + 1})$ such that $p < h^{-1}(q')$ for every $h \in \mathcal{K}$. 
    Now set $H_{n + 1} = H_n \cup \{r_{m + 1}, p\}$, and let 
    $g_{n + 1}$ and $\phi_{n + 1}$ extend the appropriate functions with $g_{n 
      + 1}(r_{m + 1}) = t$, $g_{n + 1}(p) = q$ and $\phi_{n + 1}(r_{m + 1}, h) = 
    \phi_{n + 1}(p, h) = O$ for every $h \in \mathcal{K}$. 
    
    Case 1c: there is no $p' \in H_n$ with $p' < r_{m + 1}$ but there is a 
    $p'' \in H_n$ with $r_{m + 1} < p''$. This case can be handled similarly as Case 1b.
    
    Case 1d: there is a $p' \in H_n$ with $p' < r_{m + 1}$, there is a 
    $p'' \in H_n$ with $r_{m + 1} < p''$, and for the largest such $p'$ and 
    the smallest such $p''$, there is no $h \in \mathcal{K}$ with 
    $g_n(p') < h(p')$ and $g_n(p'') > h(p'')$. In this case, let $\mathcal{K}' 
    = \{ h \in \mathcal{K} : g_n(p') > h(p') \text{ and } g_n(p'') < h(p'')\}$, 
    where $p' \in H_n$ is the largest with $p' < r_{m + 1}$ and $p'' \in 
    H_n$ is the smallest with $p'' > r_{m + 1}$. Note that $\mathcal{K}'$ may 
    be the empty set. Let $q' = g_n(p')$ and $q'' = g_n(p'')$. 
    Choose $t \in (q', q'')$ such that $t > h(r_{m + 1})$ for each  
    $h \in \mathcal{K}'$ with $h(r_{m + 1}) < q''$. Such a $t$ exists, since 
    the compactness of $\mathcal{K}'$ implies that $\{h(r_{m + 1}) : h \in 
    \mathcal{K}', h(r_{m + 1}) < q''\}$ is finite. We will set 
    $g_{n + 1}(r_{m + 1}) = t$, but we need to define the value of $g_{n + 1}$ 
    at one more place. Choose $q \in (q', t)$ with $q > h(r_{m + 1})$ for each 
    $h \in \mathcal{K}'$ with $h(r_{m + 1}) < q''$. For every $h \in 
    \mathcal{K}'$ we have $h(p') < q'$, hence also $p' < h^{-1}(q')$. 
    Therefore there is a $p \in (p', r_{m + 1})$ for which $p < h^{-1}(q')$ for 
    every $h \in \mathcal{K}'$. 
    
    Now let $H_{n + 1} = H_n \cup \{p, r_{m + 1}\}$, $g_{n + 1}$ extend $g_n$ 
    with $g_{n + 1}(p) = q$, $g_{n + 1}(r_{m + 1}) = t$. For $h \in 
    \mathcal{K}'$, either $h(r_{m + 1}) < t$ or $h(r_{m + 1}) > t$. If $h(r_{m 
      + 1}) < t$ then let $\phi_{n + 1}(r_{m + 1}, h) = \phi_n(p', h)$, if 
    $h(r_{m + 1}) > t$ then let $\phi_{n + 1}(r_{m + 1}, h) = \phi_n(p'', h)$. 
    In both cases, let $\phi_{n + 1}(p, h) = \phi_n(p', h)$. If $h \in 
    \mathcal{K} \setminus \mathcal{K}'$ then let $\phi_{n + 1}(p, h) = \phi_{n + 1}(r_{m + 1}, h) = \phi_n(p', h)$. Note that using \ref{i:s=1 <=> g < h, s=-1 <=> g > h}, $s_f(\phi_n(p', h)) = s_f(\phi_n(p'', h))$, thus \ref{i:s is alternating} implies that $\phi_n(p', h) = \phi_n(p'', h)$. All of the properties can be checked easily.
    
    Case 1e: there is a $p' \in H_n$ with $p' < r_{m + 1}$, there is a 
    $p'' \in H_n$ with $r_{m + 1} < p''$, and for the largest such $p'$ and 
    the smallest such $p''$, there is no $h \in \mathcal{K}$ with 
    $g_n(p') > h(p')$ and $g_n(p'') < h(p'')$. 
    Now let $\mathcal{K}' = \{ h \in \mathcal{K} : g_n(p') < h(p') \text{ and } 
    g_n(p'') > h(p'')\}$, where again, $p' \in H_n$ is the largest with $p' < 
    r_{m + 1}$ and $p'' \in H_n$ is the smallest with $p'' > r_{m + 1}$. Let 
    $q' = g_n(p')$ and $q'' = g_n(p'')$. The set $\{h(p'') : h \in 
    \mathcal{K}'\}$ is finite, hence there is a $t \in (q', q'')$ with $t > 
    h(p'')$ for every $h \in \mathcal{K}'$. Let $H_{n + 1} = H_n \cup 
    \{r_{m + 1}\}$, $g_{n + 1}(r_{m + 1}) = t$ and $\phi_{n + 1}(r_{m + 1}, h) 
    = \phi_n(p'', h)$ for every $h \in \mathcal{K}$. Using the fact that for no 
    $h \in \mathcal{K}$ can $h$ and any strictly increasing extension of $g_{n 
      + 1}$ have the same values on $[r_{m + 1}, p'']$, one can easily check that 
    every property is satisfied. 
    
    Using \ref{i:no crossing}, these cover all sub-cases of Case 1. Now we 
    turn to the second case. 
    
    \bigskip
    {\bf Case 2: $n = 3m + 1$.}  At this step, we make sure that $r_{m + 1} \in g_{n + 1}(H_{n + 1})$. The proofs are analogous to that of Case 1, but we include them for the sake of completeness. 
    If already $r_{m + 1} \in g_n(H_n)$ then let $H_{n + 1} = H_n$, $g_{n + 1} = g_n$ and $\phi_{n + 1} = \phi_n$. Otherwise, similarly 
    as in Case 1, there are multiple sub-cases according to the existence of 
    $q' \in g_n(H_n)$ with $q' < r_{m + 1}$, $q'' \in g_n(H_n)$ with $r_{m + 1} 
    < q''$, and whether there exists an $h \in \mathcal{K}$ such that $g_n(p') 
    < h(p')$ or $g_n(p') > h(p')$, $g_n(p'') < h(p'')$ or $g_n(p'') > h(p'')$, 
    where $p' = g_n^{-1}(q')$ and $p'' = g_n^{-1}(q'')$. Since $r_{m + 1} \in H_n$ we do not have to deal with the case $H_n = \emptyset$. 
    
    Case 2a: there is a $q' \in g_n(H_n)$ with $q' < r_{m + 1}$ but there is no 
    $q'' \in g_n(H_n)$ with $r_{m + 1} < q''$. Let $q'$ be the largest element 
    in $g_n(H_n)$, clearly $q' < r_{m + 1}$. As before, $g_n(p') < h(p')$ for 
    every $h \in \mathcal{K}$ or $g_n(p') > h(p')$ for every $h \in 
    \mathcal{K}$, where $p' = g_n^{-1}(q')$. In the first case, choose $r > p'$ and $r > h^{-1}(r_{m + 1})$ for 
    every $h \in \mathcal{K}$. Such an $r$ exists, since $\{h^{-1}(r_{m + 1}) : h \in \mathcal{K}\}$ is finite. Let $q \in (q', r_{m + 1})$ with $q < h(p')$ for every $h \in \mathcal{K}$, and choose $p \in (p', r)$ with $p > h^{-1}(r_{m + 1})$ for every $h \in \mathcal{K}$. Then let $H_{n + 1} = H_n \cup \{r, p\}$, and let $g_{n + 1}$ and $\phi_{n + 1}$ extend $g_n$ and $\phi_n$, respectively, with $g_{n + 1}(r) = r_{m + 1}$, $g_{n + 1}(p) = q$ and $\phi_{n + 1}(r, h) = \phi_{n + 1}(p, h) = \phi_n(p', h)$ for every $h \in \mathcal{K}$. 
    
    In the second case, choose $r > p'$ with $r < h^{-1}(q')$ 
    for every $h \in \mathcal{K}$. Such an $r$ exists, since for 
    every $h \in \mathcal{K}$, $h(p') < q'$ implies $p' < h^{-1}(q')$ and 
    $\{h^{-1}(q') : h \in \mathcal{K}\}$ is finite. Then set $H_{n + 1} = H_n 
    \cup \{r\}$, and let $g_{n + 1}(r) = r_{m + 1}$ and $\phi_{n + 1}(p, h) = 
    \phi_n(p', h)$ for every $h \in \mathcal{K}$. 
    
    Case 2b: there is no $q' \in g_n(H_n)$ with $q' < r_{m + 1}$ but there is a 
    $q'' \in g_n(H_n)$ with $r_{m + 1} < q''$. This case can be handled 
    similarly to Case 2a.
    
    Case 2c: there is a $q' \in g_n(H_n)$ with $q' < r_{m + 1}$, there is a 
    $q'' \in H_n$ with $r_{m + 1} < q''$, and for the largest such $q'$ and 
    the smallest such $q''$, there is no $h \in \mathcal{K}$ with 
    $g_n(p') < h(p')$ and $g_n(p'') > h(p'')$, where $p' = g_n^{-1}(q')$ and 
    $p'' = g_n^{-1}(q'')$. This is analogous to Case 1e. There exists $r \in 
    (p', p'')$ with $h^{-1}(q') > r$ for every $h \in \mathcal{K}$ such 
    that $g_n(p') > h(p')$ and $g_n(p') < h(p')$. As before, set $H_{n + 1} = 
    H_n \cup \{r\}$ and let $g_{n + 1}$ extend $g_n$ with $g_{n + 1}(r) = r_{m 
      + 1}$, and $\phi_{n + 1}$ extend $\phi_n$ with $\phi_{n + 1}(r, h) = 
    \phi_n(p', h)$ for every $h \in \mathcal{K}$. 
    
    Case 2d: there is a $q' \in g_n(H_n)$ with $q' < r_{m + 1}$, there is a 
    $q'' \in H_n$ with $r_{m + 1} < q''$, and for the largest such $q'$ and 
    the smallest such $q''$, there is no $h \in \mathcal{K}$ with 
    $g_n(p') > h(p')$ and $g_n(p'') < h(p'')$, where $p' = g_n^{-1}(q')$ and 
    $p'' = g_n^{-1}(q'')$. This is analogous to Case 1d. Let $\mathcal{K}' = \{ 
    h \in \mathcal{K} : g_n(p') < h(p') \text{ and } g_n(p'') > h(p'')\}$, 
    this may again be the empty set. Choose $r \in (p', p'')$ such that 
    $r < h^{-1}(r_{m + 1})$ for each $h \in \mathcal{K}'$ with $h^{-1}(r_{m 
      + 1}) > p'$. There is a $q \in (r_{m + 1}, q'')$ with $q > h(p'')$ for 
    every $h \in \mathcal{K}'$. Choose $p \in (r, p'')$ with $p < h^{-1}(r_{m + 
      1})$ for each $h \in \mathcal{K}'$ with $h^{-1}(r_{m + 1}) > p'$.
    
    Now let $H_{n + 1} = H_n \cup \{p, r\}$, $g_{n + 1}$ extend $g_n$ 
    with $g_{n + 1}(r) = r_{m + 1}$, $g_{n + 1}(p) = q$. For $h \in 
    \mathcal{K}'$, either $h^{-1}(r_{m + 1}) \le p'$ or $h^{-1}(r_{m + 1}) > 
    p$. If $h^{-1}(r_{m + 1}) \le p'$ then let $\phi_{n + 1}(r, h) = \phi_n(p', 
    h)$, if $h^{-1}(r_{m + 1}) > p$ then let $\phi_{n + 1}(r, h) = \phi_n(p'', 
    h)$. In both cases, let $\phi_{n + 1}(p, h) = \phi_n(p'', h)$. 
    If $h \in \mathcal{K} \setminus \mathcal{K}'$ then let $\phi_{n + 1}(p, h) 
    = \phi_{n + 1}(r, h) = \phi_n(p', h)$. 
    
    Again using \ref{i:no crossing}, these cover all sub-cases of Case 2. Now 
    we turn to the third case. 
    
    \bigskip
    {\bf Case 3: $n = 3m + 2$.} At this step, we make sure that $O_{m + 1} \in \phi_{n + 1}(H_{n + 1}, h)$ for every $h \in \mathcal{K}$. Note throughout that there is no $O \in \mathcal{O}^*_f$ with $s_f(O) = 0$. If $O_{m + 1} \in \phi_n(H_n, h)$ for any (hence, by \ref{i:phi image is same} for 
    every) $h \in \mathcal{K}$ then let $H_{n + 1} = H_n$, 
    $g_{n + 1} = g_n$ and $\phi_{n + 1} = \phi_n$. If this is not the case, 
    we consider the sub-cases according to $\phi_n(H_n, h_0)$ for a fixed $h_0 
    \in \mathcal{K}$. We suppose throughout that $s_f(O_{m + 1}) = 1$. The case 
    $s_f(O_{m + 1}) = -1$ is similar. Also, note that $H_n \neq \emptyset$, as, 
    for example, $r_1 \in H_n$. 
    
    Case 3a: $O_{m + 1} > O$ for every $O \in \phi_n(H_n, h_0)$, and for the 
    largest $O \in \phi_n(H_n, h_0)$ (with respect to $<$), $s_f(O) = -1$. 
    Let $p$ be the largest element in $H_n$, $q = g_n(p)$, then $O = \phi_n(p, 
    h_0)$. This means, using \ref{i:phi image is same} and \ref{i:s=1 <=> g 
      < h, s=-1 <=> g > h} that $\phi_n(p, h) = O$ and $g_n(p) > h(p)$ for 
    every $h \in \mathcal{K}$. Choose any $t > q$ then, as $\{h^{-1}(t) : h \in \mathcal{K}\}$ is finite, there exists $r > p$ with $r > h^{-1}(t)$ for every $h \in \mathcal{K}$. Now let $H_{n + 1} = H_n \cup \{r\}$, let $g_{n + 1}$ 
    extend $g_n$ with $g_{n + 1}(r) = t$ and let 
    $\phi_{n + 1}$ extend $\phi_n$ with $\phi_{n + 1}(r, h) = O_{m + 1}$ for 
    every $h \in \mathcal{K}$. One can easily check that the necessary 
    conditions still hold. 
    
    Case 3b: $O_{m + 1} > O$ for every $O \in \phi_n(H_n, h_0)$, and for the 
    largest $O \in \phi_n(H_n, h_0)$ (with respect to $<$), $s_f(O) = 1$. Let 
    $p$ be the largest element in $H_n$, $q = g_n(p)$, then $O = \phi_n(p, 
    h_0)$. Using that $f$ is good, there exists $O' \in \mathcal{O}^*_f$ with $O < 
    O' < O_{m + 1}$ and $s_f(O') = -1$. Now choose $r'> p$ and 
    choose $t' > q$ with $t' > h(r')$ for every $h \in 
    \mathcal{K}$. Then choose $t'' > t'$ and choose $r'' > r'$ with $r'' > h^{-1}(t'')$ for every $h \in \mathcal{K}$. 
    Now let $H_{n + 1} = H_n \cup \{r', r''\}$, and let $g_{n + 1}$ extend 
    $g_n$ with $g_{n + 1}(r') = t'$ and $g_{n + 1}(r'') = t''$, and let 
    $\phi_{n + 1}$ extend $\phi_n$ with $\phi_{n + 1}(r', h) = O'$ and $\phi_{n 
      + 1}(r'', h) = O_{m + 1}$ for every $h \in \mathcal{K}$. 
    
    The cases where $O_{m + 1} < O$ for every $O \in \phi_n(H_n, h_0)$ are 
    similar to the ones above. 
    
    Case 3c: $O_{m + 1}$ is between elements of $\phi_n(H_n, h_0)$, and if $O'$ 
    is the largest element of $\phi_n(H_n, h_0)$ with $O' < O_{m + 1}$ and 
    $O''$ is the smallest element of $\phi_n(H_n, h_0)$ with $O_{m + 1} < O''$ 
    then $s_f(O') = -1$ and $s_f(O'') = 1$. In this case, choose $O \in 
    \mathcal{O}^*_f$ with $O_{m + 1} < O < O''$ and $s_f(O) = -1$, again, such an 
    $O$ exists because $f$ is good. The orbitals $O'$ and $O''$ are 
    neighbouring ones in $\phi_n(H_n, h)$ for every $h \in \mathcal{K}$. 
    
    Notice that for every $h \in \mathcal{K}$ there exists a unique pair of 
    neighbouring points $p', p'' \in H_n$ with $\phi_n(p', h) = O'$ and 
    $\phi_n(p'', h) = O''$. Therefore, we can partition $\mathcal{K}$ into 
    finitely many compact sets according to this pair. We define $g_{n + 1}$ 
    separately on each such interval $(p', p'')$, that is, where $p'$ and $p''$ 
    are neighbouring points in $H_n$ and $\phi_n(p', h) = O'$, $\phi_n(p'', h) = 
    O''$ for some $h \in \mathcal{K}$. 
    
    So let $p', p''$ be such elements of $H_n$ and let $\mathcal{K}' = \{h \in 
    \mathcal{K} : \phi_n(p', h) = O' \text{ and } \phi_n(p'', h) = O''\}$. 
    Using the facts that $s_f(O') = -1$, $s_f(O'') = 1$ and \ref{i:s=1 <=> g < h, 
      s=-1 <=> g > h}, we have $g_n(p') > h(p')$ and $g_n(p'') < h(p'')$ for 
    every $h \in \mathcal{K}'$. Let $q' = g_n(p')$ and $q'' = g_n(p'')$, and 
    choose $q \in (q', q'')$. Let $\{r^1, r^2, \dots, r^c\} = 
    \{h^{-1}(q) : h \in \mathcal{K}'\}$, where $r^1 < r^2 < \dots < r^c$. Note 
    that $h(p') < g_n(p') = q' < q < q'' = g_n(p'') < h(p'')$ for every $h \in 
    \mathcal{K}'$, hence $p' < r^1$ and $r^c < p''$. For $1 \le j \le c$, let 
    $\mathcal{K}^j = \{h \in \mathcal{K}' : h^{-1}(q) = r^j\}$. 
    
    Choose $t \in (q', q)$ with $t > h(r^j)$ for every $1 \le j \le c$ 
    and every $h \in \mathcal{K}'$ such that $h(r^j) < q$. From now on, the 
    values of $g_{n + 1}|_{(p', p'')}$ on newly defined points will always be 
    at least $t$. This will achieve that if we add new points to take care of 
    the functions in $\mathcal{K}^j$ for some $j$, then our choices will not 
    interfere with the functions in $\mathcal{K}' \setminus \mathcal{K}^j$. 
    
    Choose $r \in (p', p'')$ with $r < h^{-1}(q')$ for every $h \in 
    \mathcal{K}'$. By setting $g_{n + 1}(r) = t$ and extending it to a strictly 
    increasing function, it can be easily seen that the extension cannot have a 
    common value with any $h \in \mathcal{K}'$ on the interval $(p', r)$. Let 
    $t^1_1 \in (t, q)$ be arbitrary and choose $r^1_1 \in (r, r^1)$ 
    with $t^1_1 < h(r^1_1)$ for every $h \in \mathcal{K}^1$. 
    Then choose $r^1_2 \in (r^1_1, r^1)$ and choose $t^1_2 \in (t^1_1, q)$ such 
    that $t^1_2 > h(r^1_2)$ for every $h \in \mathcal{K}^1$. Then let $r^1_3 = 
    r^1$ and choose $t^1_3 \in (t^1_2, q)$. 
    
    We handle the families $\mathcal{K}^j$ for $j \ge 2$ similarly. Choose 
    $t^j_1 \in (t^{j - 1}_3, q)$ and then choose $r^j_1 \in (r^{j - 1}_3, r^j)$ 
    such that $h(r^j_1) > t^j_1$ for every $h \in \mathcal{K}^j$. Then let 
    $r^j_2 \in (r^j_1, r^j)$ and choose $t^j_2 \in (t^j_1, q)$ with $t^j_2 > 
    h(r^j_2)$ for every $h \in \mathcal{K}^j$. Then let $r^j_3 = r^j$ 
    and choose $t^j_3 \in (t^j_2, q)$. 
    
    After recursively choosing the rational numbers above for every $j \le c$, 
    we choose $p \in (r^c_3, p'')$ such that $p > h^{-1}(q'')$ for 
    every $h \in \mathcal{K}'$. Now we will set $H_{n + 1} \cap (p', p'') = 
    (H_n \cap (p', p'')) \cup \{r, p, r^j_\ell : 1 \le j \le c, 1 \le \ell \le 
    3\}$. Let $g_{n + 1}$ extend $g_n$ with $g_{n + 1}(r) = t$, $g_{n + 1}(p) = 
    q$ and $g_{n + 1}(r^j_\ell) = t^j_\ell$ for every $1 \le j \le c$ and $1 
    \le \ell \le 3$. 
    Let $\phi_{n + 1}$ extend $\phi_n$ with $\phi_{n + 1}(r, h) = O'$, $\phi_{n 
      + 1}(p, h) = O''$ for every $h \in \mathcal{K}'$. Also, let $\phi_{n + 
      1}(r^j_\ell, h) = O'$ for every $\ell$ if $h \in \mathcal{K}^{j'}$ with $j' 
    > j$, and $\phi_{n + 1}(r^j_\ell, h) = O''$ if $j' < j$. If $j' = j$ then 
    let $\phi_{n + 1}(r^j_1, h) = O_{m + 1}$, $\phi_{n + 1}(r^j_2, h) = O$ and 
    $\phi_{n + 1}(r^j_3, h) = O''$. For every $h \in \mathcal{K} \setminus 
    \mathcal{K}'$, let $\phi_{n + 1}(x, h) = \phi_n(p', h) = \phi_n(p'', h)$, 
    for every $x \in\{r, p, r^j_\ell : 1 \le j \le c, 1 \le \ell \le 3\}$, 
    where we used \ref{i:s is alternating} for the last equality. 
    
    We do the same in every interval of the form $(p', p'')$, where $p'$ and 
    $p''$ are neighbours in $H_n$, and $\phi_n(h,p') = O'$ and $\phi_n(h,p'') = 
    O''$ for some $h \in \mathcal{K}$. Extending $g_{n}$ and $\phi_n$ 
    appropriately, one obtains $H_{n + 1}$, $g_{n + 1}$ and $\phi_{n + 1}$ with 
    the necessary conditions. We note that the choice of $t$ ensures that 
    condition \ref{i:extension is good} is satisfied.
    
    Case 3d: $O_{m + 1}$ is between elements of $\phi_n(H_n, h_0)$, and if $O'$ 
    is the largest element of $\phi_n(H_n, h_0)$ with $O' < O_{m + 1}$ and 
    $O''$ is the smallest element of $\phi_n(H_n, h_0)$ with $O_{m + 1} < O''$ 
    then $s_f(O') = 1$ and $s_f(O'') = -1$. This case can be handled quite 
    similarly as Case 3c. Choose $O \in \mathcal{O}^*_f$ with $O' < O < O_{m + 1}$ 
    and $s_f(O) = -1$. Again the unique pairs of neighbouring points $p', p'' \in 
    H_n$ with $\phi_n(p', h) = O'$ and $\phi_n(p'', h) = O''$ define a 
    partition of $\mathcal{K}'$. So let $p', p'' \in H_n$ be such a pair, we 
    set $q' = g_n(p')$ and $q'' = g_n(p'')$. 
    
    Let $p \in (p', p'')$ be arbitrary and let $\{t^1, \dots, t^c\} = \{h(p) : 
    h \in \mathcal{K}'\}$, where $\mathcal{K}' = \{h \in \mathcal{K} : h(p') > 
    g_n(p') \text{ and } h(p'') < g_n(p'')\}$, such that $t^1 < \dots < t^c$. 
    We set $\mathcal{K}^j = \{h \in \mathcal{K}' : h(p) = t^j\}$. Now one can 
    choose $r \in (p', p)$ with $h^{-1}(t^j) < r$ for every $h \in 
    \mathcal{K}'$ and $1 \le j \le c$ if $h^{-1}(t^j) < p$. 
    Let $t \in (q', t^1)$ be such that $t < h(p')$ for every $h \in 
    \mathcal{K}'$. Now suppose that for $j' < j$ and $1 \le \ell \le 3$ the 
    points $r^{j'}_\ell$ and $t^{j'}_\ell$ are given. Then choose $r^j_1$ 
    arbitrarily for the set $(r, p)$ if $j = 1$ and from $(r^{j - 1}_3, p)$ if 
    $j > 1$. Then choose $t^j_1$ from $(t, t^j)$ if $j = 1$ and from $(t^{j - 
      1}_3, t^j)$ if $j > 1$ such that $h(r^j_1) < t^j_1$ for every $h \in 
    \mathcal{K}^j$. Then choose $t^j_2 \in (t^j_1, t^j)$ and choose $r^j_2 \in 
    (r^j_1, p)$ such that $h(r^j_2) > t^j_2$ for every $h \in \mathcal{K}^j$. 
    Finally, choose $r^j_3 \in (r^j_2, p)$ and set $t^j_3 = t^j$. 
    
    After recursively choosing the points $r^j_\ell$ and $t^j_\ell$, choose $q 
    \in (t^c, q'')$ such that $q > h(p'')$ for every $h \in \mathcal{K}'$. As 
    before, let $H_{n + 1} \cap (p', p'') = (H_n \cap (p', p'')) \cup \{r, p, 
    r^j_\ell : 1 \le j \le c, 1 \le \ell \le 3\}$, and define $g_{n + 1}(r) = 
    t$, $g_{n + 1}(p) = q$ and $g_{n + 1}(r^j_\ell) = t^j_\ell$ for every $1 
    \le j \le c$ and $1 \le \ell \le 3$. For $h \in \mathcal{K}^j$ let $\phi_{n 
      + 1}(r, h) = O'$, $\phi_{n + 1}(r^{j'}_\ell, h) = O'$ for every $j' < j$ 
    and $1 \le \ell \le 3$, $\phi_{n + 1}(r^j_1, h) = O$, $\phi_{n + 1}(r^j_2, 
    h) = O_{m + 1}$, $\phi_{n + 1}(r^j_3, h) = O''$, and $\phi_{n + 
      1}(r^{j''}_\ell, h) = \phi_{n + 1}(p, h) = O''$ for every $j'' > j$, $1 \le 
    \ell \le 3$. For every $h \in \mathcal{K} \setminus \mathcal{K}'$ we set 
    $\phi_{n + 1}(x, h) = \phi_{n}(p', h)$ for every $x \in (H_{n + 1} \cap 
    (p', p'')) \setminus H_n$. 
    
    It is straightforward to check that $H_{n + 1}$, $g_{n + 1}$ and $\phi_{n + 
      1}$ obtained in this way satisfy the conditions.
  \end{proof}

  Now we show that the inverse of the function constructed above translates $\mathcal{K}$ into $\mathcal{C}$. 

  \begin{proof}[Proof of Proposition \ref{p:cpct translated into conj class of fixed-point free}]
    It is clear using \ref{i:extension} and \ref{i:surjective} that $g = \bigcup_n g_n : \Q \to \Q$ and $\phi = \bigcup_n \phi_n : \Q \times \mathcal{K} \to \mathcal{O}^*_f$ are surjective functions. Using \ref{i:g_n increasing} and \ref{i:phi increasing}, $g$ is strictly increasing (hence $g \in \aut(\Q, <)$) and $\phi(., h)$ is increasing for every $h \in \mathcal{K}$. 
    
    We now show that $f$, $g^{-1}h$ and $\phi(., h)$ satisfy the conditions of 
    Lemma \ref{l:conjugacy in aut(Q)} to prove that $f$ and $g^{-1}h$ are 
    conjugate automorphisms for every $h \in \mathcal{K}$. We start by showing 
    that $g^{-1}h$ is good for every $h \in \mathcal{K}$. First of all, it has no fixed point, since $f$ has no fixed point and therefore \ref{i:s=1 <=> g < h, s=-1 <=> g > h} covers all cases, hence $g_n(p) \neq h(p)$ for any $p \in \Q$. 
    Now suppose towards a contradiction that there are distinct orbitals $O_1, 
    O_2 \in \mathcal{O}^*_{g^{-1}h}$ such that either $s_{g^{-1}h}(O_1) = 
    s_{g^{-1}h}(O_2) = 1$ or $s_{g^{-1}h}(O_1) = s_{g^{-1}h}(O_2) = -1$ and 
    there is no orbital $O_3 \in \mathcal{O}^*_{g^{-1}h}$ with $s_{g^{-1}h}(O_3) \neq s_{g^{-1}h}(O_1)$ between them. We suppose for the rest of the proof that $s_{g^{-1}h}(O_1) = s_{g^{-1}h}(O_2) = 1$, the case when they equal $-1$ is analogous. Note that in this case, 
    \begin{equation}
    \label{e:eq1}
    g(p) < h(p) \text{ for every } p \in (O_1 \cup O_2).
    \end{equation}
    
    Let $p' \in O_1$ and $p'' \in O_2$ be arbitrary. Then $\phi(p', h), \phi(p'', h) \in \mathcal{O}^*_f$. We consider the following two cases separately. 
    
    Case 1: $\phi(p', h) = \phi(p'', h)$. Let $O = \phi(p', h)$. Then, using 
    the fact that $\phi(., h)$ is increasing provided by \ref{i:phi increasing}, $\phi(p, h) = O$ 
    for every $p \in (p', p'')$. Using \eqref{e:eq1} and \ref{i:s=1 <=> g < 
      h, s=-1 <=> g > h}, $s_f(\phi(p', h)) = s_f(O) = 1$. Hence if $p \in (p', 
    p'')$ then $g(p) < h(p)$ using \ref{i:s=1 <=> g < h, s=-1 <=> g > h} and 
    the fact that $\phi(p, h) = O$. Let $n$ be large enough such that $p', p'' 
    \in H_n$ and let $\{r^1, \dots, r^m\} = H_n \cap [p', p'']$ where $p' = 
    r^1 < \dots < r^m = p''$. Then applying \ref{i:extension is good} to each 
    of the intervals $[r^j, r^{j + 1}]$, the facts that $h(r^j) > g_n(r^j)$ 
    and $h(r^{j + 1}) > g_n(r^{j + 1})$ imply $h(r) \ge g_n(r^{j + 1})$ for 
    every $r \in [r^j, r^{j + 1}]$. It follows (since $g$ is an increasing 
    extension of $g_n$) that $g^{-1}h(r^1) \ge r^2$. Using 
    induction, one can show with the same argument that $(g^{-1}h)^{m - 1}(r^1) 
    \ge r^m$, hence $(g^{-1}h)^{m - 1}(p') \ge p''$. This fact implies that 
    $p'$ and $p''$ are in the same orbital with respect to $g^{-1}h$, 
    contradicting our assumption. 
    
    Case 2: $\phi(p', h) \neq \phi(p'', h)$. Again using \eqref{e:eq1} and 
    \ref{i:s=1 <=> g < h, s=-1 <=> g > h} twice, $g(p') < h(p')$ and $g(p'') 
    < h(p'')$, hence $s_f(\phi(p', h)) = s_f(\phi(p'', h)) = 1$. Using the fact 
    that $f$ is good, there is $O \in \mathcal{O}^*_f$ between $\phi(p', h)$ and 
    $\phi(p'', h)$ with $s_f(O) = -1$, since there is no fixed point between 
    $O_1$ and $O_2$. Using that $\phi(., h)$ is increasing and surjective 
    provided by \ref{i:phi increasing} and \ref{i:surjective}, there is $p 
    \in (p', p'')$ with $\phi(p, h) = O$. 
    Then \ref{i:s=1 <=> g < h, s=-1 <=> g > h} ensures that $g(p) > h(p)$, 
    hence $g^{-1}h(p) < p$, therefore there exists $O' \in 
    \mathcal{O}^*_{g^{-1}h}$ with $O_1 < O' < O_2$ and $s_{g^{-1}h}(O') = -1$, 
    contradicting our assumptions. This completes the proof of the fact that 
    $g^{-1}h$ is good. 
    
    The function $\phi(., h) : \Q \to \mathcal{O}^*_f$ is increasing 
    and surjective using its construction and \ref{i:phi increasing}, 
    \ref{i:surjective}. Condition \ref{il:g(q) > q <=> s(phi(q)) = 1} and \ref{il:g(q) < q <=> s(phi(q)) = -1} of Lemma \ref{l:conjugacy in aut(Q)} follows easily from \ref{i:s=1 <=> g < h, s=-1 <=> g > h} and the surjectivity of $\phi(., h)$. 
    
    Therefore the conditions of Lemma \ref{l:conjugacy in aut(Q)} 
    are satisfied for $f$, $g^{-1}h$ and $\phi(., h)$, hence $f$ and $g^{-1}h$ are conjugate automorphisms for every $h \in \mathcal{K}$. This completes the proof of the proposition. 
  \end{proof}

  \begin{proof}[Proof of Theorem \ref{t:autQ}]
    First we show the ``only if'' part. Using Corollary \ref{c:fininf}, for almost every
    $f \in \aut(\Q, <)$, $\fix(f)$ is finite. Since the 
    cardinality of fixed points is the same for conjugate automorphisms, 
    it is clear that the conjugacy class of $f$ can only be non-Haar null if 
    $\fix(f)$ is finite. 
    
    The property that between any two distinct orbitals $O_1, O_2 \in 
    \mathcal{O}^*_f$ with either $s_f(O_1) = s_f(O_2) = 1$ or $s_f(O_1) = s_f(O_2) 
    = -1$, there exists an orbital $O_3 \in \mathcal{O}^*_f$ with $s_f(O_3) \neq 
    s_f(O_1)$, is also conjugacy invariant. Hence it is enough to prove the 
    following lemma to finish the ``only if'' part of the theorem.
    
    \begin{lemma}
      \label{l:prevalent aut(Q)}
      For almost every $f \in \aut(\Q, <)$, for distinct orbitals 
      $O_1, O_2 \in \mathcal{O}^*_f$ with $O_1 < O_2$ such that either $s_f(O_1) = 
      s_f(O_2) = 1$ or $s_f(O_1) = s_f(O_2) = -1$, there exists an orbital 
      $O_3 \in \mathcal{O}^*_f$ with $O_1 < O_3 < O_2$ and $s_f(O_3) \neq 
      s_f(O_1)$. 
    \end{lemma}
    \begin{proof}
      Let $\mathcal{H}$ be the set of those automorphisms that satisfy the 
      property in the lemma, and let $\mathcal{A}$ denote the set 
      of functions $f \in \mathcal{H}^c$ such that 
      we can choose distinct orbitals $O_1^f, O_2^f \in \mathcal{O}^*_f$ 
      such that $O^f_1 < O^f_2$, $s_f(O_1^f) = s_f(O_2^f) = 1$, and between 
      $O_1^f$ and $O_2^f$, there is no orbital $O_3 \in \mathcal{O}^*_f$ with 
      $s_f(O_3) \neq 1$. 
      Also, let us denote by $\mathcal{A}'$ the set of functions $f \in 
      \mathcal{H}^c$, such that we can choose distinct orbitals $O_1^f, O_2^f 
      \in \mathcal{O}^*_f$ such that $O^f_1 < O^f_2$, $s_f(O_1^f) = s_f(O_2^f) = 
      -1$, and between $O_1^f$ and $O_2^f$, there is no orbital $O_3 \in 
      \mathcal{O}^*_f$ with $s_f(O_3) \neq -1$. 
      We show that $\mathcal{A}$ is Haar null and the same can be proved 
      similarly for $\mathcal{A}'$. Since it is easy to see that $\mathcal{H}^c = 
      \mathcal{A} \cup \mathcal{A}'$, proving this will finish the proof of the 
      lemma.
      
      We use Theorem \ref{t:null} to show that the conjugacy invariant set 
      $\mathcal{A}$ is Haar null. Let $(p_0, q_0), (p_1, q_1), \dots$ be an 
      enumeration of all pairs $(p, q)$ 
      with $p < q$, and for all $n \in \omega$, let 
      \begin{equation*}
      \begin{split}
      \mathcal{A}_n =& \{f \in \aut(\Q, <) : \text{$p_n$ and $q_n$ are in distinct 
        orbitals with respect to $f$} \\ 
      &\text{and $f(r) > r$ for every $r \in [p_n, 
        q_n]$}\} \\
      =& \bigcap_{k \in \Z} \bigcap_{r \in [p_n, q_n]} \{f \in \aut(\Q, <) : 
      \text{$f^k(p_n) < q_n$ and $f(r) > r$}\}.
      \end{split}
      \end{equation*}
      Note that $\mathcal{A} = \bigcup_{n \in \omega} \mathcal{A}_n$ and 
      $\mathcal{A}_n$ is Borel for every $n \in \omega$. 
      Using Theorem \ref{t:null}, it is enough to show that there is a 
      conjugacy invariant set $\mathcal{B}$ with $1 = \id_{\Q} \in 
      \overline{\mathcal{B}}$ and $\mathcal{B} \cap \bigcup_{n \in \omega} 
      \mathcal{A}_n^{-1}\mathcal{A}_n = \emptyset$. Let 
      \begin{equation*}
      \begin{split}
      \mathcal{B} = \{ & f \in \aut(\Q, <) : \\ &\text{$\fix(f)$ is finite, 
        $|\mathcal{O}^*_f| = 2|\fix(f)| + 1$ and $f(r) \ge r$ for every $r \in 
        \Q$}\}.
      \end{split}
      \end{equation*}
      Note that the condition $|\mathcal{O}^*_f| = 2|\fix(f)| + 1$ essentially 
      states that between neighbouring fixed points, every point is in the same 
      orbital (roughly speaking, this means that there are no ``irrational fixed 
      points''). 
      
      It is easy to see that $\mathcal{B}$ is a conjugacy invariant set, 
      and also that $\id_\Q \in \overline{\mathcal{B}}$. 
      Let $n \in \omega$ be arbitrary, it remains to show that if $f$, $g \in 
      \mathcal{A}_n$ then $f^{-1}g \not \in \mathcal{B}$. Let $O$ be the orbit of 
      $p_n$ with respect to $g$ and let $I = \{r \in \Q : r \le r' \text{ for 
        some $r' \in O$}\}$. Note that $I$ is convex, and since $g(r) > r$ for 
      every $r \in (p_n, q_n)$ but $p_n$ and $q_n$ are in different orbitals 
      (with respect to both $f$ and $g$), $q_n \not \in I$. 
      
      There are two cases with respect to the relationship of $I$ and the 
      orbitals of $f$. Suppose first that $I$ does not split orbitals of $f$, 
      that is, there is no $r \in I$ and $k \in \Z$ such that $f^k(r) \not \in 
      I$. Then the sets $I$ and $\Q \setminus I$ are invariant under both $f$ and 
      $g$ (and $f^{-1}$ and $g^{-1}$), thus $I$ does not split any orbitals of 
      $f^{-1}g$. 
      Moreover, $I$ has no greatest element, nor $\Q \setminus I$ has a least 
      element, since any such element would need to be a fixed point of $g$, but 
      $g$ does not have a fixed point in the interval $(p_n, q_n)$. Now suppose 
      that $f^{-1}g \in \mathcal{B}$. Then it has a greatest fixed point (if any) 
      that belongs to $I$ and a least fixed point (if any) that belongs to $\Q 
      \setminus I$, hence between the two, every point is in the same orbital. 
      This contradicts the fact that $I$ does not split the orbitals of $f^{-1}g$.
      
      Now suppose that $I$ splits an orbital of $f$, thus there exist 
      $r \in I$ and $k \in \Z$ such that $f^k(r) \not \in I$. Since $f(r) > r$ 
      for every $r \in (p_n, q_n)$, it follows that there is an $r \in (p_n, 
      \infty) \cap I$ such that $f(r) \not \in I$. Then $g^{-1}(f(r)) \not \in 
      I$, since $I$ does not split orbitals of $g$. By setting $r' = 
      g^{-1}(f(r))$, we see that $f^{-1}g(r') = r \in I$, thus $f^{-1}g(r') < 
      r'$, hence $f^{-1}g \not \in \mathcal{B}$ also in this case, finishing the 
      proof of the lemma. 
    \end{proof}
    
    Now we prove the ``if'' part of the theorem. Let $f$ be a good automorphism. We already showed in Proposition \ref{p:cpct translated into conj class of fixed-point free} that if $f$ has no fixed points then its conjugacy class is compact catcher. Therefore, using also Fact \ref{f:biter}, it is enough to show that $\mathcal{C}$, the conjugacy class of $f$ is compact biter. 
    
    To show this, take a non-empty compact set $\mathcal{F} \subset \aut(\Q, <)$. In our proof, we partition $\Q$ 
    into finitely many intervals bounded by the fixed points of $f$, and 
    on each interval, we define a suitable part of an automorphism $g$, towards showing that $\mathcal{C}$ is compact biter. We construct $g$ using Proposition \ref{p:cpct translated into conj class of fixed-point free}.
    
    Let $\{p_1, p_2, \dots, p_{k-1}\}$ be the set of fixed points of $f$ 
    (which is necessarily finite) with $p_1 < p_2 < \dots < p_{k-1}$. We now 
    choose $q_1, q_2, \dots, q_{k-1} \in \Q$ such that if we set $\mathcal{U} = 
    \{h \in \aut(\Q, <) : \forall i (h(p_i) = q_i)\}$ then $\mathcal{U} \cap 
    \mathcal{F} \neq \emptyset$. Let $\mathcal{K} = \mathcal{U} \cap 
    \mathcal{F}$. Note that $\mathcal{U}$ is closed, so $\mathcal{K}$ is compact, and $\mathcal{U}$ is also open, so it is enough to construct an automorphism $g$ with $\mathcal{K} \subset g\mathcal{C}$ to finish the proof of the theorem. 
    
    Let us use the notation $p_0 = q_0 = -\infty$ and $p_k = q_k = +\infty$. For each $i < k$ let $\pi_i : (p_i, p_{i + 1}) \to \Q$, $\rho_i : (q_i, q_{i + 1}) \to \Q$ be increasing bijections, $\mathcal{K}_i = \{h \restriction (p_i, p_{i + 1}) : h \in \mathcal{K}\}$ and $f_i = f \restriction (p_i, p_{i + 1})$. Then $\rho_i \mathcal{K}_i \pi_i^{-1} \subset \aut(\Q, <)$ is compact, and one can easily see that $\pi_i f_i \pi_i^{-1}$ is a good automorphism without fixed points for each $i < k$. Using Proposition \ref{p:cpct translated into conj class of fixed-point free} there are automorphisms $g_i\in \aut(\Q, <)$ for each $i < k$ such that 
    \begin{equation}
      \label{e:g_i K_i subset C_i}
      g_i^{-1}\rho_i \mathcal{K}_i \pi_i^{-1} \subset \mathcal{C}_i,
    \end{equation}
    where $\mathcal{C}_i$ is the conjugacy class of $\pi_i f_i \pi_i^{-1}$. Let 
    $$g = \{(p_i, q_i) : 1 \le i \le k - 1\} \cup \bigcup_{i < k} \rho_i^{-1}g_i\pi_i.$$
    It is clear that $g \in \aut(\Q, <)$, we now show that $g^{-1} \mathcal{K} \subset \mathcal{C}$ to finish the proof of the theorem. 
    
    Let $h \in \mathcal{K}$ and $h_i = h \restriction (p_i, p_{i + 1})$. Using \eqref{e:g_i K_i subset C_i}, we can find an automorphism $k_i \in \aut(\Q, <)$ for each $i < k$  with 
    \begin{equation}
      \label{e:k_i conjugates}
      g_i^{-1}\rho_i h_i \pi_i^{-1} = k_i^{-1}\pi_i f_i \pi_i^{-1}k_i.
    \end{equation}
    Let
    $$k = \{(p_i, p_i) : 1 \le i \le k - 1\} \cup \bigcup_{i < k} \pi_i^{-1}k_i\pi_i,$$
    it is enough to show that $g^{-1}h = k^{-1} f k$. Let $p \in \Q$. If $p = p_i$ for some $i < k$, $i \ge 1$, then $(g^{-1}h)(p) = p = (k^{-1}fk)(p)$. Otherwise, choose $i < k$ with $p \in (p_i, p_{i + 1})$. Then using \eqref{e:k_i conjugates},
    \begin{align*}
      (g^{-1}h)(p) &= (\pi_i^{-1}g_i^{-1}\rho_ih_i)(p) = (\pi_i^{-1}g_i^{-1}\rho_ih_i\pi_i^{-1}\pi_i)(p) = (\pi_i^{-1}k_i^{-1}\pi_i f_i \pi_i^{-1}k_i \pi_i)(p) \\
      &= (k^{-1}fk)(p).
    \end{align*}

    And thus the proof of the theorem is complete.  
  \end{proof}
  \begin{proof}[Proof of Theorem \ref{t:autQ union of null}]
    Using Theorem \ref{t:autQ}, the union of the Haar null conjugacy classes is exactly the set of the automorphisms with infinitely many fixed points together with those that violate the condition of Lemma \ref{l:prevalent aut(Q)}. The former set is Haar null using Theorem \ref{t:FACP <-> finfin -> infinf}, and the latter is Haar null by Lemma \ref{l:prevalent aut(Q)}. Hence the union of the two is also Haar null. 
  \end{proof}

  \bigskip
  \textbf{Acknowledgements.} We would like to thank to R. Balka, Z. Gyenis, A. Kechris, C. Rosendal, S. Solecki and P. Wesolek for many valuable remarks and discussions.
  
  \bibliographystyle{apalike}
  \bibliography{autQ}

\end{document}